\documentclass[12pt,a4paper]{amsart}
\usepackage{graphics}
\usepackage{epsfig}
\usepackage{graphicx}
\theoremstyle{plain}
\usepackage{amssymb}

\advance\hoffset-20mm \advance\textwidth40mm

\newtheorem{theorem}{Theorem}
\newtheorem{lemma}{Lemma}
\newtheorem*{theo*}{Theorem}
\newtheorem{proposition}{Proposition}

\theoremstyle{definition}

\newtheorem*{definition*}{Definition}

\newtheorem{remark}{Remark}

\def\KK{{\Bbb K}}
\def\sl{{\mathfrak sl}}
\def\ad{\mathop{\rm ad}}
\def\Ker{{\rm Ker}}
\def\Der{{\rm Der}}
\begin{document}
\sloppy
\title[Polynomial Lie algebras of rank one]
{Finite-dimensional subalgebras \\ in polynomial Lie algebras of rank one}

\author%
{I.V.\ Arzhantsev, E.A.\ Makedonskii, A.P.\ Petravchuk}
\address{Ivan V. Arzhantsev:\
Department of Higher Algebra, Faculty of Mechanics and Mathematics,
Moscow State University, Leninskie Gory 1, GSP-1, Moscow, 119991,
Russia } \email{arjantse@mccme.ru}
\address{Eugeni A. Makedonskii:
Algebra Department, Faculty of Mechanics and Mathematics, Kyiv
Taras Shevchenko University, 64, Volodymyrskaia street, 01033 Kyiv, Ukraine}
\email{makedonskii$\_$e@mail.ru}
\address{Anatoliy P. Petravchuk:
Algebra Department, Faculty of Mechanics and Mathematics, Kyiv
Taras Shevchenko University, 64, Volodymyrskaia street, 01033 Kyiv, Ukraine}
\email{aptr@univ.kiev.ua}

\thanks{
The first author was supported by the RFBR grant
09-01-90416-Ukr-f-a. The third author was supported by the DFFD,
grant F28.1/026}
\begin{abstract}
Let $W_n(\KK)$ be the Lie algebra of derivations of the polynomial
algebra \linebreak $\KK[X]:=\KK[x_1,\ldots,x_n]$ over an
algebraically closed field $\KK$ of characteristic zero. A
subalgebra $L\subseteq W_n(\KK)$ is called polynomial if it is a
submodule of the $\KK[X]$-module $W_n(\KK)$. We prove that the
centralizer of every nonzero element in $L$ is abelian provided
$L$ has rank one. This allows to classify finite-dimensional
subalgebras in polynomial Lie algebras of rank one.
\end{abstract}
\maketitle


\section*{Introduction}

Let $\KK$ be an algebraically closed field of characteristic zero and
$\KK[X]:=\KK[x_1,\ldots,x_n]$ the polynomial algebra over $\KK$. Recall that
a {\it derivation} of $\KK[X]$ is a linear operator \linebreak $D \colon \KK[X] \to \KK[X]$
such that
$$
D(fg) \, = \, D(f)g \, + \, fD(g) \quad \text {for all} \quad f,g \in \KK[X].
$$
Every derivation of the algebra $\KK[X]$ has the form
$$
P_1\frac{\partial}{\partial x_1}+\ldots+P_n\frac{\partial}{\partial x_n} \quad
\text{for some} \quad P_1,\ldots,P_n \in \KK[X].
$$
A derivation $D$ may be extended to the derivation $\overline{D}$
of the field of rational functions \linebreak $\KK(X):=\KK(x_1,\ldots,x_n)$ by
$$
\overline{D}\left(\frac{f}{g}\right):=\frac{D(f)g-fD(g)}{g^2}.
$$
The kernel $S$ of $\overline{D}$ is an algebraically closed subfield of $\KK(X)$,
cf.~\cite[Lemma~2.1]{NN}.

Denote by $W_n(\KK)$ the Lie algebra of all derivations of
$\KK[X]$ with respect to the standard commutator. The study of the
structure of the Lie algebra $W_n(\KK)$ and of its subalgebras is
an important problem appearing in various contexts (note that in
case $\KK =\mathbb{R}$ or $\KK =\mathbb{C}$ we have the Lie
algebra $W_n(\KK)$ of all vector fields with polynomial
coefficients on $\mathbb{R}^{n}$ or $\mathbb{C}^{n}$). Since
$W_n(\KK)$ is a free $\KK[X]$-module (with the basis
$\frac{\partial}{\partial x_1},\ldots,\frac{\partial}{\partial
x_n}$), it is natural to consider the subalgebras $L\subseteq
W_n(\KK)$ which are $\KK[X]$-submodules. Following the work of
V.M.~Buchstaber and D.V.~Leykin~\cite{BL}, we call such
subalgebras the {\it polynomial Lie algebras}. In~\cite{BL}, the
polynomial Lie algebras of maximal rank were considered. Earlier,
D.A.~Jordan studied  subalgebras of the Lie algebra $\Der(R)$ for
a commutative ring $R$ which are $R$-submodules in the $R$-module
$\Der(R)$ (see \cite{J1}).

 In this note, we study  polynomial Lie algebras $L$ of rank one. In Section~2 we prove
that the centralizer of every nonzero element in $L$ is abelian.
Clearly, this property is inherited by any subalgebra in $L$. It
is not difficult to describe all finite-dimensional Lie algebras
with this property, see~Proposition~\ref{P1}. In Theorem~1 we give
a classification of finite-dimensional subalgebras in polynomial
Lie algebras of rank one: every such subalgebra is either abelian,
or solvable with an abelian ideal of codimension one and trivial
center, or isomorphic to $\sl_2(\KK)$. Moreover, for all these
three types we construct an explicit realization in some $L$.
Applying obtained results to the Lie algebra $W_{1}(\KK )$ we give
a  description of all finite dimensional subalgebras of
$W_{1}(\KK )$ (Proposition~\ref{Lie}). In case $\KK =\mathbb{C}$
this description can be easily deduced from classical results of S.~Lie
(see \cite{Lie}) about realizations (up to local diffeomorphisms)
of finite dimensional Lie algebras by vector fields on the complex
line. In \cite{Lie}, S.~Lie has also classified analogous
realizations on the complex plane and on the real line. On the
real plane such a classification is given in \cite{Olver}.


\section{Lie algebras with abelian centralizers}

We begin with an elementary lemma on submodules of a free module. Let $A$ be a unique
factorization domain and $N=Ae_1\oplus\ldots\oplus Ae_n$ a free $A$-module. An element
$x\in N$ is said to be {\it reduced} if the condition $x=ax'$ with $a\in A$ and $x'\in N$
implies that the element $a$ is invertible in $A$.

\begin{lemma} \label{L1}
For every submodule $M\subseteq N$ of rank one there exist an
ideal $I\subseteq A$ and a reduced element $m_0\in N$ such that
$M=Im_0$. The submodule $M$ defines the element $m_0$ uniquely up
to multiplication by an invertible element of $A$.
\end{lemma}

\begin{proof}
Take a nonzero element $m\in M$, $m=a_1e_1+\ldots+a_ne_n$. Let $a$ be the greatest common divisor
of $a_1,\ldots,a_n$, and $m_0=a_1^0e_1+\ldots+a_n^0e_n$, where $a_i^0=a_i/a$. Since $M$ has
rank one, for every nonzero $m'\in M$ there are nonzero $c,d \in A$ such that
$cm+dm'=0$. Then $acm_0+dm'=0$. If $m'=a_1'e_1+\ldots+a_n'e_n$, then $aca_i^0+da_i'=0$
for all $i=1,\ldots,n$. If $d$ does not divide $ac$, then some prime $p\in A$
divides all the elements $a_1^0,\ldots,a_n^0$. But the elements $a_1^0,\ldots,a_n^0$
are coprime, a contradiction.  Thus $m'$ equals $bm_0$ with $b=ac/d$. This proves that
all elements of $M$ have the form $bm_0$ for some $b\in A$. Clearly, all elements $b\in A$
such that $bm_0 \in M$ form an ideal $I$ of $A$. The second assertion follows from the fact
that a free $A$-module has no torsion.
\end{proof}

We say that a derivation
$P_1\frac{\partial}{\partial x_1}+\ldots+P_n\frac{\partial}{\partial x_n}$
is {\it reduced} if the polynomials $P_1,\ldots,P_n$ are
coprime. Setting $A=\KK[X]$ and $N=W_n(\KK)$, we get the following variant of Lemma~\ref{L1}.

\begin{lemma} \label{L2}
For every submodule $M\subseteq W_n(\KK)$ of rank one there exist
an ideal $I\subseteq \KK[X]$ and a reduced derivation $D_0\in
W_n(\KK)$ such that $M=ID_0$. The submodule $M$ defines the
derivation $D_0$ uniquely up to nonzero scalar.
\end{lemma}

Now we study the centralizers of elements in a polynomial Lie algebra of rank one.

\begin{proposition} \label{P0}
Let $L$ be a subalgebra of the Lie algebra $W_n(\KK)$. Assume that
$L$ is a submodule of rank one in the $\KK[X]$-module $W_n(\KK)$.
Then the centralizer of any nonzero element in $L$ is abelian.
\end{proposition}

\begin{proof}
By Lemma~\ref{L2}, the subalgebra $L$ has the form $ID_0$ for some
reduced derivation $D_0 \in W_n(\KK)$. Denote by $\overline{D_0}$
the extension of $D_0$ to the field $\KK(X)$, and let $S$ be the
kernel of $\overline{D_0}$. Take any nonzero element $fD_0 \in L$,
$f\in I$, and consider its centralizer $C=C_L(fD_0)$. For every
nonzero element $gD_0\in C$ one has
$$
[fD_0,gD_0] = (fD_0(g)-gD_0(f))D_0=0.
$$
This implies $D_0(f)g-fD_0(g)=0$, thus $\overline{D_0}(f/g)=0$ and
$f/g \in S$. Take another nonzero element $hD_0\in C$. By the same
arguments we get $f/h \in S$. This shows that $g/h \in S$. The
latter condition is equivalent to $[gD_0,hD_0]=0$, so the
subalgebra $C$ is abelian.
\end{proof}

The next proposition seems to be known, but having no precise reference we
supply it with a complete proof. By $Z(F)$ we denote the center of a Lie algebra $F$.

\begin{proposition} \label{P1}
Let $F$ be a finite-dimensional Lie algebra over an algebraically closed field $\KK$
of characteristic zero. Assume that the centralizers of all nonzero elements in $F$
are abelian. Then either $F$ is abelian, or $F \cong A \leftthreetimes \langle b\rangle$,
where $b \in F$, $A \subset F$ is an abelian ideal and $Z(F)=0$, or $F \cong \sl_2(\KK)$.
\end{proposition}

\begin{proof}
If the centralizers of all nonzero elements of a Lie algebra $F$
is abelian, then the same property holds for every subalgebra of
$F$. Assume that $F$ is not abelian and the centralizers of all
elements of $F$ are abelian. Then the center $Z(F)$ is trivial.

\smallskip

{\it Case 1.}\ $F$ is solvable. Then $F$ contains a non-central
one-dimensional ideal $\langle a \rangle$, see \cite[II.4.1,
Corollary~B]{Hum}. Let $A$ be the centralizer of $a$ in $F$.
Clearly, $A$ is an abelian ideal of codimension one in $F$. Then
$F \cong A \leftthreetimes \langle b\rangle$ for any $b \in F
\setminus A$.

\smallskip

{\it Case 2.}\ $F$ is semisimple. Then $F = F_1 \oplus \ldots
\oplus F_k$ is the sum of simple ideals. Since the centralizer of
every element $x\in F_1$ contains $F_2 \oplus \ldots \oplus F_k$,
we conclude that $F$ is simple. Let $H$ be a Cartan subalgebra in
$F$ and $F=N_-\oplus H \oplus N_+$ the Cartan decomposition with
opposite maximal nilpotent subalgebras $N_-$ and $N_+$ in $F$, see
\cite[II.8.1]{Hum}. Since the centrilizer of every element in
$N_+$ is abelian, either the subalgebra $N_+$ is abelian or
$Z(N_+)=0$. The second possibility is excluded because $N_+$ is
nilpotent. Thus $N_+$ is abelian. This is the case if and only if
the root system of the Lie algebra $F$ has rank one, or,
equivalently, $F  \cong \sl_2(\KK)$.

\smallskip

{\it Case 3.}\ $F$ is neither solvable nor semisimple. Consider
the Levi decomposition $F = R \leftthreetimes G$, where $G$ is a
maximal semisimple subalgebra and $R$ is the radical of $F$. By
Case 2, the algebra $G$ is isomorphic to $\sl_2(\KK)$. Denote by
$A$ the ideal of $R$ which coincides with $R$ if $R$ is abelian,
and $A=[R,R]$ otherwise. By Case 1, the ideal $A$ is abelian.
Consider the decomposition $A = A_1 \oplus \ldots \oplus A_s$ into
simple $G$-modules with respect to the adjoint representation. If
$\dim A_1=1$, then the centralizer of a nonzero element in $A_1$
contains $G$, a contradiction. Suppose that $\dim A_1 \ge 2$. Fix
an $\sl_2$-triple $\{e,h,f\}$ in $G$ and take a highest vector
$x\in A_1$ with respect to the Borel subalgebra $\langle
e,h\rangle$. Then $[e,x]=0$ and the centralizer $C_F(x)$ contains
the subalgebra $A \leftthreetimes \langle e\rangle$. The latter is
not abelian because the adjoint action of the element $e$ on $A_1$
is not trivial. This contradiction concludes the proof.
\end{proof}


\section{Main results}

In this section we get a classification of finite-dimensional subalgebras
in polynomial Lie algebras of rank one.

\begin{theorem} \label{T1}
Let $L$ be a polynomial Lie algebra of rank one in $W_n(\KK)$, where
$\KK$ is an algebraically closed field of characteristic zero, and $F\subset L$
a finite-dimensional subalgebra. Then one of the following conditions holds.
\begin{enumerate}
\item [(1)] $F$ is abelian; \item [(2)] $F \cong A \leftthreetimes
\langle b\rangle$, where $A \subset F$ is an abelian ideal and
$[b,a]=a$ for every $a\in A$; \item [(3)] $F$ is a
three-dimensional simple Lie algebra, i.e., $F\cong \sl_2(\KK)$.
\end{enumerate}
\end{theorem}

\begin{proof}
By Propositions~\ref{P0} and \ref{P1}, every finite-dimensional
subalgebra $F\subset L$ is either abelian, or has the form $A
\leftthreetimes \langle b\rangle$, or is isomorphic to
$\sl_2(\KK)$. It remains to prove that in the second case we may
find $b\in F$ with $[b,a]=a$ for every $a\in A$. Take any element
$b$ with $F=A \leftthreetimes \langle b\rangle$.

Let us prove that the operator $\ad(b)$ is diagonalizable. Assuming the converse, let
$a_0,a_1 \in A$ be nonzero elements with $[b,a_1]=\lambda a_1 +a_0$,
$[b,a_0]=\lambda a_0$ for some
$\lambda \in \KK$. By Lemma~\ref{L2}, the subalgebra $L$ has the form $ID_0$ for some
ideal $I\subseteq \KK[X]$ and some reduced derivation $D_0\in W_n(\KK)$. Set
$$
a_0=fD_0, \quad a_1=gD_0, \quad b=hD_0, \quad f,g,h \in I.
$$
The relations $[b,a_1]=\lambda a_1 +a_0$, $[b,a_0]=\lambda a_0$, and $[a_0,a_1]=0$
are equivalent to
$$
hD_0(g)-gD_0(h)= \lambda g + f, \quad
hD_0(f)-fD_0(h)= \lambda f, \quad
fD_0(g)-gD_0(f)=0.
$$
Multiplying the second relation by $g$, we get
$$
hgD_0(f)-fgD_0(h)= \lambda fg.
$$
This and the third relation imply
$$
hfD_0(g)-fgD_0(h)= \lambda fg \quad \Rightarrow \quad hD_0(g)-gD_0(h)= \lambda g.
$$
Together with the first relation it gives $f=0$, a contradiction.

\smallskip

Now assume that $[b,a_1]=\lambda_1a_1$ and $[b,a_2]=\lambda_2a_2$ for some
$\lambda_1,\lambda_2\in\KK$. If
$a_1=fD_0, \\ a_2=gD_0, b=hD_0$, then we obtain the relations
$$
hD_0(f)-fD_0(h)= \lambda_1 f, \quad
hD_0(g)-gD_0(h)= \lambda_2 g, \quad
fD_0(g)-gD_0(f)=0.
$$
Consequently,
$$
ghD_0(f)=gf(\lambda_1+D_0(h))=fhD_0(g)=fg(\lambda_2+D_0(h)).
$$
This proves that $\lambda_1=\lambda_2$ and hence $\ad(b)$ is a
scalar operator. Since $F$ is not abelian, $\ad(b)$ is nonzero
and, multiplying $b$ by a suitable scalar, we may assume that
$\ad(b)$ is the identical operator. This completes the proof of
Theorem~\ref{T1}.
\end{proof}

Let us show that all three possibilities indicated in Theorem~\ref{T1}
are realizable. Take a derivation $D_0 \in W_n(\KK)$ such that there exist
non-constant polynomials $p,q \in \KK[X]$ with $D_0(p)=0$ and $D_0(q)=1$.
For example, one may take $D_0= \frac{\partial}{\partial x_2}+
P_3\frac{\partial}{\partial x_3}+\ldots+P_n\frac{\partial}{\partial x_n}$
with arbitrary $P_3,\ldots,P_n\in\KK[X]$, and $p=x_1$, $q=x_2$.

\smallskip

The subalgebra $\langle D_0, pD_0,\ldots,p^{m-1}D_0\rangle$ is an
$m$-dimensional abelian subalgebra in $\KK[X]D_0$ for every
positive integer $m$.

\smallskip

The subalgebra $A \leftthreetimes \langle b\rangle$
with $\dim A=m$ may be obtained
by setting $A= \langle D_0, pD_0,\ldots,p^{m-1}D_0\rangle$ and $b=-qD_0$.
Indeed,
$$
[-qD_0,f(p)D_0]=(-D_0(f(p))+f(p)D_0(q))D_0=f(p)D_0 \quad
\text{for every} \quad f(p) \in \KK[p].
$$

\smallskip

Finally, the derivations $e=q^2D_0$, $h=2qD_0$ and $f=-D_0$ form
an $\sl_2$-triple in $\KK[X]D_0$.

\begin{remark}
\label{R1}
The structure of finite-dimensional subalgebras in a polynomial
Lie algebra $L=ID_0$ depends on properties of the derivation $D_0$.
In particular, if $\Ker(\overline{D_0})=\KK$, then all abelian
subalgebras in $\KK[X]D_0$ are one-dimensional.
\end{remark}

Our last result concerns finite-dimensional subalgebras in the Lie algebra $W_1(\KK)$.
By Lemma~\ref{L2}, every polynomial Lie algebra in $W_1(\KK)$ has the form
$L=q(x)\KK[x]\frac{\partial}{\partial x}$ with some polynomial $q(x)\in\KK[x]$.

\begin{proposition}\label{Lie}
Let $L=q(x)\KK[x]\frac{\partial}{\partial x}$ be a polynomial
algebra.
\begin{enumerate}
\item If $\deg q(x) \ge 2$, then every finite dimensional Lie
subalgebra in $L$ is one-dimensional. \item If $\deg q(x)=1$, then
every finite dimensional Lie subalgebra in $L$ is either
one-dimensional or coincides with $F_k=\langle
q(x)\frac{\partial}{\partial x}, q(x)^k\frac{\partial}{\partial
x}\rangle$ for some $k\ge 2$. \item If $q(x)=\text{const}\ne 0$
(i.e. $L=W_{1}(\KK$)), then every finite dimensional Lie
subalgebra in $L$ is either one-dimensional, or coincides with
$F_{k,\beta}=\langle (x+\beta)\frac{\partial}{\partial x},
(x+\beta)^k\frac{\partial}{\partial x}\rangle$ for some $\beta \in
\KK$ and $k=0,2,3,\dots$, or is a three-dimensional subalgebra
$$
F(\beta) \, = \, \langle \frac{\partial}{\partial x}, \, (x+\beta)\frac{\partial}{\partial x},
\, (x+\beta)^2\frac{\partial}{\partial x} \rangle, \quad \text{where} \quad \beta \in \KK.
$$
\end{enumerate}
\end{proposition}

\begin{proof}
Let us describe all two-dimensional subalgebras in $W_1(\KK)$.
Every such subalgebra has the form
$$
\langle f(x)\frac{\partial}{\partial x}, \, g(x)\frac{\partial}{\partial x} \rangle
\quad \text{with} \quad f(x),\, g(x) \in \KK[x] \quad \text{and} \quad fg'-f'g=g.
\quad (*)
$$
If $\deg(f)\ge 2$, then looking at the highest terms of $fg'$ and $f'g$,
we get $\deg(f)=\deg(g)$.
But the polynomials $(f+\lambda g, g)$ satisfy relation (*) for every $\lambda \in \KK$,
and thus we may assume that $f$ is linear. Each root of $g$ is also a root of $f$,
so $g$ is proportional to $f^k$ for some $k=0,2,3,\ldots$.
This observation together with Theorem~\ref{T1} and Remark~\ref{R1} proves all the assertions.
\end{proof}
If we consider obtained in Proposition~\ref{Lie} realizations up
to automorphisms of the polynomial ring $\KK[x]$, then in case $\deg
q(x)=1$ for the Lie algebra $F_{k}$ one can take $q(x)=x$, and in case
$q(x)=\text{const}\ne 0$ one can take $\beta =0.$

%
\end{document}